\documentclass[letterpaper]{amsart}
\usepackage[margin=1in]{geometry} 
\usepackage[utf8]{inputenc}
\usepackage{amsmath,amsthm,amssymb,amsfonts,mathrsfs}
\usepackage{dsfont}
\usepackage{bookmark}
\usepackage{xcolor}

\numberwithin{equation}{section}
    
\newtheorem{theorem}{Theorem}[section]
\newtheorem{lemma}[theorem]{Lemma}
\newtheorem{proposition}[theorem]{Proposition}

\theoremstyle{definition}

\theoremstyle{remark}
\newtheorem{remark}[theorem]{Remark}

\title[Resolvent Restriction Estimates]{$L^p - L^q$ resolvent restriction estimates for submanifolds}
\author[M. D. Blair]{Matthew D. Blair}
\author[C. Park]{Chamsol Park}
\address{Department of Mathematics and Statistics, University of New Mexico, Albuquerque, NM 87131, USA}
\email{blair@math.unm.edu}
\address{Department of Mathematics, Johns Hopkins University, Baltimore, MD 21218, USA}
\email{cspark@jhu.edu}
\date{}

\usepackage{pgfplots}
\pgfplotsset{compat=1.16}

\begin{document}

\begin{abstract}
    We consider restriction analogues on hypersurfaces of the uniform Sobolev inequalities in Kenig, Ruiz, and Sogge \cite{KenigRuizSogge1987UniformSobolev} and the resolvent estimates in Dos Santos Ferreira, Kenig, and Salo \cite{DSFKenigSalo2014Forum}.
\end{abstract}

\maketitle

\section{Introduction}

The purpose of this paper is to find restriction analogues of the work in Kenig, Ruiz, and Sogge \cite{KenigRuizSogge1987UniformSobolev} and Dos Santos Ferreira, Kenig, and Salo \cite{DSFKenigSalo2014Forum}. We denote by $H^{2, p}$ the space of functions with two derivatives in $L^p (\mathbb{R}^n)$. We write the flat Laplacian on $\mathbb{R}^n$ as $\Delta$. One of the results in \cite{KenigRuizSogge1987UniformSobolev} was that given $\delta>0$, for all $z\in \mathbb{C}$ with $|z|\geq \delta$, we have that
\begin{align}\label{KRS estimate}
    \|u\|_{L^q (\mathbb{R}^n)}\leq C \|(\Delta+z)u\|_{L^p (\mathbb{R}^n)},\quad u\in H^{2, p}(\mathbb{R}^n),
\end{align}
where
\begin{align}\label{KRS p, q conditions}
    \frac{1}{p}-\frac{1}{q}=\frac{2}{n},\quad \text{and } \min\left(\left|\frac{1}{p}-\frac{1}{2} \right|, \left|\frac{1}{q}-\frac{1}{2} \right| \right)>\frac{1}{2n}.
\end{align}
The estimate \eqref{KRS estimate} was also studied further on the flat torus in Shen \cite{Shen2001AbsoluteContinuity}, and Shen and Zhao \cite{ShenZhao2008UniformSobolevInequalities}.

As observed in \cite[\S 3]{KenigRuizSogge1987UniformSobolev}, the estimate of the form \eqref{KRS estimate} is helpful in unique continuation results for the solutions to the Schr\"oringer equations. Later, it was discovered that the estimate of the type of \eqref{KRS estimate} can be applied to various other circumstances (see, e.g., Goldberg and Schlag \cite{GoldbergSchlag2004LimitingAbsorption}, Ionescu and Schlag \cite{IonescuSchlag2006AgmonKatoKuroda}, Frank \cite{Frank2011EigenvalueBoundsWithComplexPotentials}, Frank and Sabin \cite{FrankSabin2017RestrictionTheorem}, and references therein).

Let $\Sigma$ be a smooth hypersurface in $\mathbb{R}^n$. We can compute $\|u\|_{L^q (\Sigma)}$ as in the Stein-Tomas restriction theorem (cf. \cite[Corollary 2.2.2]{Sogge1993fourier}). Indeed, suppose that $\Sigma$ is a smooth hypersurface in $\mathbb{R}^n$. For $\tilde{\beta}\in C_0^\infty (\mathbb{R}^n)$, we consider a compactly supported measure
\begin{align}\label{Compact induced measure}
    d\mu(x)=\tilde{\beta}(x)d\sigma(x),\quad x\in \mathbb{R}^n,
\end{align}
where $d\sigma(x)$ is the surface measure on $\Sigma$. We can then write
\begin{align}\label{Restriction Lq u computation}
    \|u\|_{L^q (\Sigma)}=\left(\int_\Sigma |u(x)|^q\:d\mu \right)^{\frac{1}{q}}=\left(\int_\Sigma |u(x)|^q\tilde{\beta}(x)\:d\sigma(x) \right)^{\frac{1}{q}}.
\end{align}
With this in mind, we first show a restriction analogue of \eqref{KRS estimate}.

\begin{theorem}\label{Thm Euclidean resolvent restriction estimates for hypersurfaces}
    Let $n\geq 3$. Suppose $\Sigma$ is a smooth hypersurface in $\mathbb{R}^n$. If
    \begin{align}\label{Exponent condition for hypersurfaces in Rn}
        \begin{split}
            \frac{n}{p}-\frac{n-1}{q}=2 \text{ and } \begin{cases}
                \frac{2(n-1)^2}{n^2-3n+4}<q<\frac{2(n-1)}{n-3}, & \text{when } n=3, 4, 5, \\
                \frac{2n^2-5n+4}{n^2-4n+8}<q<\frac{2(n-1)}{n-3}, & \text{when } n\geq 6,
            \end{cases}
        \end{split}
    \end{align}
    then, given $\delta>0$, for all $z\in \mathbb{C}$ with $|z|\geq \delta>0$,
    \begin{align}\label{Unif Sobolev est for hypsurf}
        \|u\|_{L^q (\Sigma)}\leq C_\delta \|(-\Delta+z)u\|_{L^p (\mathbb{R}^n)},\quad u\in H^{2, p} (\mathbb{R}^n),
    \end{align}
    where $\|u\|_{L^q(\Sigma)}$ is computed as in \eqref{Restriction Lq u computation}.
\end{theorem}

If $(M, g)$ is a compact Riemannian manifold without boundary of dimension $n$, then $\Delta_g$ denotes the Laplace-Beltrami operator, or Laplacian in short, associated with the metric $g$. Unless otherwise specified, we write throughout this paper
\begin{align*}
    P=\sqrt{-\Delta_g}.
\end{align*}
If $e_\lambda$ is the eigenfunction of the Laplacian associated with the eigenvalue $\lambda$, we write
\begin{align}\label{e lambda eigfcn}
    -\Delta_g e_\lambda=\lambda^2 e_\lambda.
\end{align}
The eigenvalue $\lambda$ are discrete in that
\begin{align*}
    0\leq \lambda_0\leq \lambda_1 \leq \cdots,\quad \text{where } -\Delta_g e_j (x)=\lambda_j^2 e_j (x)
\end{align*}
so that the only limit point of the set of all eigenvalues is $+\infty$. Let $E_j$ be the projection onto the $j$-th eigenspace defined by
\begin{align*}
    E_j f(x)=\left(\int_M f(y) \overline{e_j (y)}\:dV_g \right) e_j (x).
\end{align*}
Given $\delta>0$, we set, as in \cite{DSFKenigSalo2014Forum},
\begin{align*}
    \Xi_\delta&=\{z\in \mathbb{C}\setminus \mathbb{R}_-: \mathrm{Re}\sqrt{z}\geq \delta\} \\
    &=\{z\in \mathbb{C}\setminus \mathbb{R}_-: (\mathrm{Im}\:z)^2\geq 4\delta^2 (\delta^2-\mathrm{Re}\sqrt{z})\}.
\end{align*}
If $(M, g)$ is a compact Riemannian manifold without boundary of dimension $n\geq 3$ and $\delta\in (0, 1)$, Dos Santos Ferreira, Kenig, and Salo \cite[Theorem 1.1]{DSFKenigSalo2014Forum} showed that there exists a uniform constant $C>0$ such that for all $u\in C^\infty (M)$ and all $z\in \Xi_\delta$, we have
\begin{align}\label{DKS resolvent estimate}
    \|u\|_{L^{\frac{2n}{n-2}}(M)}\leq C \|(\Delta_g-z)u\|_{L^{\frac{2n}{n+2}}(M)}.
\end{align}
Later, Bourgain, Shao, Sogge, and Yao \cite{BourgainShaoSoggeYao2015Resolvent} showed that there is a relationship between uniform spectral projection estimates and uniform resolvent estimates. Specifically, they showed that when $(M, g)$ is a compact Riemannian manifold without boundary of dimension $n\geq 3$ and $0<\epsilon(\lambda)\leq 1$ decreases monotonically to zero as $\lambda \to +\infty$ satisfying $\epsilon(2\lambda)\geq \frac{1}{2}\epsilon(\lambda)$ with $\lambda\geq 1$, one has the uniform spectral projection estimates
\begin{align*}
    \left\|\sum_{|\lambda-\lambda_j|\leq \epsilon(\lambda)} E_j f \right\|_{L^{\frac{2n}{n-2}}(M)}\leq C \epsilon(\lambda) \lambda \|f\|_{L^{\frac{2n}{n+2}}(M)},\quad \lambda\geq 1,
\end{align*}
if and only if one has the uniform resolvent estimates
\begin{align*}
    & \|u\|_{L^{\frac{2n}{n-2}}(M)}\leq C\|(\Delta_g+(\lambda+i\mu)^2)u\|_{L^{\frac{2n}{n+2}}(M)}, \\
    & \lambda, \mu\in \mathbb{R},\; \lambda\geq 1,\; |\mu|\geq \epsilon(\lambda),\; u\in C^\infty (M).
\end{align*}
For more details about the relationship between spectral projection estimates and resolvent estimates, we refer the reader to Cuenin \cite{Cuenin2024FromSpectralCluster}, and references therein. In the proof of this result, the estimate \eqref{DKS resolvent estimate} was revisited (see \cite[Theorem 2.5]{BourgainShaoSoggeYao2015Resolvent}). Our next result is a restriction analogue of \cite[Theorem 1.1]{DSFKenigSalo2014Forum} and \cite[Theorem 2.5]{BourgainShaoSoggeYao2015Resolvent} as follows.

\begin{theorem}\label{Thm Manifold resolvent restriction estimates for hypersurface}
    Let $(M, g)$ be a compact Riemannian manifold without boundary of dimension $n\geq 3$. Given $\delta>0$, we set
    \begin{align}\label{z lambda mu setup}
        z=\lambda+i\mu,\quad \lambda, \mu \in \mathbb{R},\quad \lambda\geq \delta,\; |\mu|\geq \delta.
    \end{align}
    Let $\Sigma$ be a hypersurface of $M$. If we assume
    \begin{align}\label{(p, q) condition for hypsurf in mfld}
        \frac{n}{p}-\frac{n-1}{q}=2,\; 3\leq n\leq 12,\text{ and } \frac{2n}{n-1}\leq q\leq \frac{2(n-1)(n+1)}{n^2-n-4},
    \end{align}
    then
    \begin{align}\label{Manifold resolvent est for hypsurf}
        \|u\|_{L^q (\Sigma)}\leq C_\delta \|(\Delta_g+z^2)u\|_{L^p (M)}.
    \end{align}
\end{theorem}

\begin{remark}
    The arguments in Theorem \ref{Thm Euclidean resolvent restriction estimates for hypersurfaces}-\ref{Thm Manifold resolvent restriction estimates for hypersurface} follow from the arguments in \cite{BourgainShaoSoggeYao2015Resolvent}, \cite{DSFKenigSalo2014Forum}, and the computation on hypersurfaces in \cite[\S 5]{BlairPark2024RestrictionOfSchrodingerEigenfunctions}. If we use the arguments in \cite{BourgainShaoSoggeYao2015Resolvent} and \cite{DSFKenigSalo2014Forum} with the computation on codimension $2$ submanifolds in \cite[\S 6]{BlairPark2024RestrictionOfSchrodingerEigenfunctions}, we can also obtain the analogous results of Theorem \ref{Thm Euclidean resolvent restriction estimates for hypersurfaces}-\ref{Thm Manifold resolvent restriction estimates for hypersurface} on codimension $2$ submanifolds. In \eqref{Restriction Lq u computation}, we considered the Lebesgue measure on a hypersurface, but we can also consider the induced Lebesgue measure on an $(n-2)$-dimensional submanifold $\Sigma$, and so, similarly we can compute $\|u\|_{L^q (\Sigma)}$. In this paper, for simplicity we focus only on the hypersurface cases, and we only state the analogous results of Theorem \ref{Thm Euclidean resolvent restriction estimates for hypersurfaces}-\ref{Thm Manifold resolvent restriction estimates for hypersurface} for codimension $2$ submanifolds as follows:
    \begin{theorem}
        Let $n\geq 3$. Suppose $\Sigma$ is an $(n-2)$-dimensional smooth submanifold of $\mathbb{R}^n$. If $\frac{n}{p}-\frac{n-2}{q}=2,\; n\geq 3,\; \text{ and }\; \frac{2(n-2)^2}{n^2-5n+8}<q<\frac{2(n-2)}{n-3}$, then given $\delta>0$, for all $z\in \mathbb{C}$ with $|z|\geq \delta>0$,
        \begin{align}\label{Unif Sobolev est for codim 2}
        \|u\|_{L^q (\Sigma)}\leq C_\delta \|(-\Delta+z)u\|_{L^p (\mathbb{R}^n)},\quad u\in H^{2, p} (\mathbb{R}^n),
        \end{align}
        where $\|u\|_{L^q(\Sigma)}$ is computed as in \eqref{Restriction Lq u computation}.
    \end{theorem}
    
    \begin{theorem}
        Let $(M, g)$ be a compact Riemannian manifold without boundary of dimension $n\geq 3$. Given $\delta>0$, we set $z=\lambda+i\mu$ as in \eqref{z lambda mu setup}. Let $\Sigma$ be an $(n-2)$-dimensional submanifold of $M$. If we assume $\frac{n}{p}-\frac{n-2}{q}=2,\; 3\leq n\leq 4,\text{ and } \frac{2(n-2)^2}{n^2-5n+8}<q\leq \frac{2(n+1)(n-2)}{n^2-n-4}$, then
        \begin{align}\label{Manifold resolvent est for codim 2}
            \|u\|_{L^q (\Sigma)}\leq C_\delta \|(\Delta_g+z^2)u\|_{L^p (M)}.
        \end{align}
    \end{theorem}
    
\end{remark}

For notations, the constants $C$ are uniform with respect to the frequency $\lambda$, but may depend on manifolds $M$, submanifold $\Sigma$, a smooth bump function $\tilde{\beta}$ used in \eqref{Restriction Lq u computation}, and exponent $p, q$. The constants $C$ may also be different line by line, but each of the constants are different only up to some uniform constant. We write $A\lesssim B$ when there is a uniform constant $C>0$ such that $A\leq CB$. We also write $A\approx B$, if $A\lesssim B$ and $B\lesssim A$.

\subsection*{Outline of the work}
In \S \ref{S: Proof of Thm for Euclidean case}, we prove Theorem \ref{Thm Euclidean resolvent restriction estimates for hypersurfaces}. To change resolvent problems to oscillatory integral problems, we shall use the tools in \cite{KenigRuizSogge1987UniformSobolev}, \cite{Sogge1988concerning}, and \cite{DSFKenigSalo2014Forum}. The argument in Hu \cite{Hu2009lp} utilizing the results in Greenleaf and Seeger \cite{GreenleafSeeger1994fourier} as in \cite{BlairPark2024RestrictionOfSchrodingerEigenfunctions} will be helpful to estimate oscillatory integral operators. In \S \ref{S: Proof of Thm for manifolds}, we prove Theorem \ref{Thm Manifold resolvent restriction estimates for hypersurface}. Estimating oscillatory integral operators will be quite similar to what we will do in \S \ref{S: Proof of Thm for Euclidean case} in the sense of \cite{Hu2009lp} and \cite{GreenleafSeeger1994fourier}. The difference is that we will use the wave kernel expression developed by Bourgain, Shao, Sogge, and Yao \cite{BourgainShaoSoggeYao2015Resolvent} to change our problem to oscillatory integral problems.

\subsection*{Acknowledgement}
The authors thank Christopher Sogge for helpful discussions about the uniform Sobolev estimates.

\section{Proof of Theorem \ref{Thm Euclidean resolvent restriction estimates for hypersurfaces}}\label{S: Proof of Thm for Euclidean case}
In this section, to prove Theorem \ref{Thm Euclidean resolvent restriction estimates for hypersurfaces}, we follow the arguments in \cite{DSFKenigSalo2014Forum} (see also \cite{KenigRuizSogge1987UniformSobolev}), but instead of using Stein's oscillatory integral theorem for the Carleson-Sj\"olin condition (cf. \cite[\S 2.2]{Sogge1993fourier}, \cite[\S 2]{DSFKenigSalo2014Forum}, \cite{Stein1993HarmonicAnalysisRealvariable}, and so on), we shall use uniform Sobolev trace type estimates from \cite{BlairPark2024RestrictionOfSchrodingerEigenfunctions} which are obtained from the results in \cite{GreenleafSeeger1994fourier} and \cite{Hu2009lp}.

We consider the principal branch of the square root $\sqrt{z}$ on $\mathbb{C}\setminus \mathbb{R}_-$, with the convention $\arg z\in (-\pi, \pi)$ for $z\in \mathbb{C}\setminus \mathbb{R}_-$. We remark that by a limiting argument (cf. \cite[p.338]{KenigRuizSogge1987UniformSobolev}), we can focus only on $z$ satisfying $|z|\geq \delta>0$ and $\mathrm{Im}(z)\not=0$. By this, we may be able to circumvent the branch cut which corresponds to the cases where $-z$ is in the spectrum of $\Delta$, and so, we do not need to consider the branch cut of $\sqrt{z}$. Indeed, recall that, for $f\in \mathcal{S}(\mathbb{R}^n)$, the work of \cite{KenigRuizSogge1987UniformSobolev} showed \eqref{KRS estimate} by proving
\begin{align}\label{KRS (2.16)}
    \left\|\left(\frac{\hat{f}(\xi)}{|\xi|^2-z} \right)^\vee \right\|_{L^{\frac{2n}{n-2}}(\mathbb{R}^n)}\lesssim \|f\|_{L^{\frac{2n}{n+2}}(\mathbb{R}^n)},\quad |z|\geq \delta>0,\quad \mathrm{Im}(z)\not=0,
\end{align}
and so, in our computation as well, it is natural to consider the distributional formula
\begin{align}\label{D' formula}
    (|\xi|^2-(\lambda+i0)^2)^{-1}=\mathrm{p.v.}(|\xi|^2-\lambda^2)^{-1}+i\pi \delta_0 (|\xi|^2-\lambda^2),
\end{align}
where $\delta_0 (|\xi|^2-\lambda^2)$ is a pull-back of a distribution (cf. \cite[A.4]{Sogge2014hangzhou}, \cite[\S 7.2]{FriedlanderJoshi1998IntroToDistributions}, and so on). By \eqref{D' formula}, for $z\in \mathbb{R}$, if we sum the limits of $(|\xi|^2-(z\pm i\epsilon)^2)^{-1}$ in the upper and lower half planes by taking $\epsilon\to 0+$, then $i\pi \delta_0(|\xi|^2-z^2)$ can be canceled out. Indeed, we may have that
\begin{align*}
    \sum_\pm (|\xi|^2-(z \pm i0))^{-1}=2\mathrm{p.v.} (|\xi|^2-z^2)^{-1}.
\end{align*}
By this, it is sufficient to assume $\mathrm{Im}(z)\not=0$ in our proof of \eqref{Unif Sobolev est for hypsurf} as well.

With this in mind, we define $F_0$ by
\begin{align*}
    F_0 (|w|, z)=(2\pi)^{-n} \int_{\mathbb{R}^n} \frac{e^{iw\cdot \xi}}{|\xi|^2+z}\:d\xi,\quad z\in \mathbb{C},\; w\in \mathbb{R}^n,\; |z|\geq \delta>0,\; \mathrm{Im}(z)\not=0.
\end{align*}
It is known that $F_0$ is a fundamental solution of $-\Delta+z$, i.e.,
\begin{align}\label{F0 fundamental solution}
    (-\Delta+z) F_0=\delta_0
\end{align}
where $\delta_0$ is the delta distribution. We recall some known results of $F_0$. For details, we refer to \cite[Lemma 4.3]{Sogge1988concerning}, \cite[Lemma 3.1]{DSFKenigSalo2014Forum}, \cite{KenigRuizSogge1987UniformSobolev}, and the references therein. 

\begin{proposition}\label{Prop summary of F0}
    There is a constant $C>0$ such that
    \begin{align*}
        |F_0 (r, z)|\leq Cr^{-n+2},\quad \text{when } r\leq |z|^{-\frac{1}{2}}.
    \end{align*}
    Moreover, we have
    \begin{align*}
        F_0 (r, z)=|z|^{\frac{n-1}{4}-\frac{1}{2}} e^{-\sqrt{z}r} r^{-\frac{n-1}{2}} a(r, z),\quad \text{when } r\geq |z|^{-\frac{1}{2}},
    \end{align*}
    where the amplitude $a$ has the following size estimates
    \begin{align*}
        \left|\frac{\partial^\alpha}{\partial r^\alpha} a(r, z) \right|\leq C_\alpha r^{-|\alpha|},\quad \text{for } \alpha\in \mathbb{N} \text{ and } r\geq |z|^{-\frac{1}{2}}. 
    \end{align*}
    The constants are all uniform with respect to $z\in \mathbb{C}$, where $|z|\geq \delta>0$.
\end{proposition}

Now, without loss of generality, we can parametrize a smooth hypersurface $\Sigma$ in $\mathbb{R}^n$ as
\begin{align}\label{Hypsurf param in Rn}
    x_n=h(x'),\quad x=(x', x_n)\in \mathbb{R}^{n-1}\times \mathbb{R},\quad |x_n|\ll 1,\quad \text{where } h:\mathbb{R}^{n-1}\to \mathbb{R} \text{ is smooth.}
\end{align}
Indeed, for elements in a sufficiently fine partition of unity, this can be arranged by applying Euclidean isometries which preserve $L^p$ norms. We define the operator $T(z)$ by
\begin{align*}
    T(z)u(x')=\int_{\mathbb{R}^n} \tilde{F_0} (x', y, z) u(y)\:dy,
\end{align*}
where
\begin{align*}
    \tilde{F_0} (x', y, z)=F_0 (|(x', h(x'))-y|, z)\tilde{a}(x')=F_0 (|x-y|, z)\tilde{a}(x'),
\end{align*}
and $\tilde{a}(x')$ is a smooth function depending on $q$, the volume element, and $\tilde{\beta}\in C_0^\infty$ in \eqref{Restriction Lq u computation} when we consider the $L^q$ norms of $T(z)u(x')$ later, compared to the $L^q$ norms of $T(z)u(x)$ defined analogously. In this section, when we write $|x-y|$, we think of it as
\begin{align*}
    |x-y|=|(x', h(x'))-y|,
\end{align*}
unless otherwise specified.

We now consider a partition of unity (cf. \cite[p.828]{DSFKenigSalo2014Forum}). Let $\beta_0\in C_0^\infty (\mathbb{R})$ be such that
\begin{align*}
    \mathrm{supp}(\beta_0) \subset [-1, 1],\quad \beta_0 \equiv 1 \text{ on } [-1/2, 1/2],
\end{align*}
and
\begin{align*}
    \beta=\beta_0 (\cdot/2)-\beta_0\in C_0^\infty (\mathbb{R}),\quad \mathrm{supp}(\beta)\subset [-2, -1/2]\cup [1/2, 2],\quad \beta_j=\beta(2^{-j}\cdot) \text{ for } j\geq 1
\end{align*}
so that we have a partition of unity
\begin{align*}
    \beta_0 (t)+\sum_{j=1}^\infty \beta(2^{-j}t)=\sum_{j=0}^\infty \beta_j (t)=1,\quad t\in \mathbb{R},
\end{align*}
Using this partition of unity, we decompose $T(z)$ dyadically
\begin{align*}
    T(z)=T_0 (z)+\sum_{j=1}^\infty T_j (z),
\end{align*}
where
\begin{align}\label{Tj construction}
    T_j (z) u(x')=\int_{\mathbb{R}^n} \beta_j (|z|^{\frac{1}{2}} |x-y|) \tilde{F_0} (x', y, z)u(y)\:dy,\quad j\geq 0.
\end{align}

We first consider $T_0 (z)$ defined in \eqref{Tj construction}. We first have the following version of the Sobolev trace formula.

\begin{lemma}\label{Lemma Sobolev trace for hypsurf}
    Suppose $\frac{n}{p}-\frac{n-1}{q}=2$, $1<p, q<\infty$, and $n\geq 3$. Then we have
    \begin{align}\label{T0 est for hypsurf in Euclidean}
        \|T_0 (z) f\|_{L^q (\Sigma)}\lesssim \|f\|_{L^p(\mathbb{R}^n)}.
    \end{align}
\end{lemma}

We note that the kernel $T_0 (z) (x', y)$ of $T_0 (z)$ is bounded by a constant times
\begin{align*}
    |x-y|^{-n+2} \mathds{1}_{|x-y|\leq |z|^{-\frac{1}{2}}} (x, y),\quad x, y\in \mathbb{R}^n,\quad \text{where } |x-y|=|(x', h(x'))-y|.
\end{align*}
We write $y\in \mathbb{R}^n$ as
\begin{align*}
    y=(y', y_n)\in \mathbb{R}^{n-1}\times \mathbb{R}.
\end{align*}
By this and the parametrization in \eqref{Hypsurf param in Rn}, we have
\begin{align*}
    |T_0 (z) (x', y)|\lesssim k(x', y', y_n):=(|x'-y'|+|h(x')-y_n|)^{2-n}.
\end{align*}
With this in mind, we recall \cite[Proposition 5.2]{BlairPark2024RestrictionOfSchrodingerEigenfunctions}.
\begin{proposition}[\cite{BlairPark2024RestrictionOfSchrodingerEigenfunctions}]\label{BP Prop 5.2}
    Suppose $\frac{n}{p}-\frac{n-1}{q}=2$, $1<p, q<\infty$, and $n\geq 3$. We write coordinates in $\mathbb{R}^n$ as $(y', s)\in \mathbb{R}^{n-1} \times \mathbb{R}$. Define
    \begin{align*}
        k(x', y', s)=(|x'-y'|+|s|)^{2-n},\quad \text{where } x'\in \mathbb{R}^{n-1},
    \end{align*}
    Then the operator
    \begin{align*}
        Tf(x'):=\int k(x', y', s) f(y', s)\:ds\:dy'
    \end{align*}
    defines a bounded linear map $T:L^p (\mathbb{R}^n) \to L^q (\mathbb{R}^{n-1})$.
\end{proposition}
Then the proof of Proposition \ref{BP Prop 5.2} gives us Lemma \ref{Lemma Sobolev trace for hypsurf}. For details, see \cite[Proposition 5.2]{BlairPark2024RestrictionOfSchrodingerEigenfunctions}. We now consider the other $T_j (z)$ with $j\geq 1$. We want to show that, if $j\geq 1$, then for some $\alpha_n (p, q)>0$,
\begin{align}\label{WTS in S2}
    \|T_j (z)u\|_{L^q (\Sigma)}\lesssim (2^j)^{-\alpha_n (p, q)} \|u\|_{L^p (\mathbb{R}^n)},\quad \text{when } n\geq 3 \text{ and \eqref{Exponent condition for hypersurfaces in Rn} holds}
\end{align}
so that summing up these estimates for all $j\geq 1$ yields \eqref{Unif Sobolev est for hypsurf}. The kernel of $T_j (z)$ is supported where
\begin{align}\label{x-y range}
    2^{j-1} |z|^{-\frac{1}{2}}\leq |x-y|\leq 2^{j+1} |z|^{-\frac{1}{2}},\quad x=(x', h(x'))\in \Sigma,\; y\in \mathbb{R}^n,
\end{align}
by the support properties of $\beta_j (|z|^{\frac{1}{2}} |x-y|)$ in \eqref{Tj construction}. By the construction \eqref{Tj construction} and Proposition \ref{Prop summary of F0}, we write
\begin{align*}
    T_j (z) u(x')=|z|^{\frac{n-1}{4}-\frac{1}{2}}\int_{\mathbb{R}^n} e^{-\sqrt{z}|x-y|} \frac{a(x, y, z)}{|x-y|^{\frac{n-1}{2}}} \beta(2^{-j}|z|^{\frac{1}{2}}|x-y|)u(y)\:dy,\quad j\geq 1,
\end{align*}
where, abusing notations, $a(x, y, z)$ is $a(|x-y|, z)$ in Proposition \ref{Prop summary of F0}. We consider two cases separately:
\begin{align*}
    \arg z\in \left[-\frac{\pi}{2}, \frac{\pi}{2}\right] \quad \text{and} \quad \arg z\not\in\left[-\frac{\pi}{2}, \frac{\pi}{2} \right],
\end{align*}
where we use the principal branch for $z\in \mathbb{C}$. We first suppose $\arg z\in \left[-\frac{\pi}{2}, \frac{\pi}{2}\right]$. We note that
\begin{align*}
    \mathrm{Re}\sqrt{z}=|z|^{\frac{1}{2}}\cos\left(\frac{\arg z}{2} \right),\quad \mathrm{Im}\sqrt{z}=|z|^{\frac{1}{2}} \sin\left(\frac{\arg z}{2} \right).
\end{align*}
By this and $\arg z\in \left[-\frac{\pi}{2}, \frac{\pi}{2} \right]$, we have
\begin{align}\label{e to the sqrt of z times |x-y|}
    \begin{split}
        |e^{-\sqrt{z}|x-y|}|&=|e^{-(\mathrm{Re}\sqrt{z}+i \mathrm{Im}\sqrt{z})|x-y|}| \\
        &=e^{-\mathrm{Re}\sqrt{z}|x-y|}=e^{-|z|^{\frac{1}{2}}\cos \left(\frac{\arg z}{2} \right)|x-y|} \leq e^{-|z|^{\frac{1}{2}}\cos \frac{\pi}{4}|x-y|}=e^{-\frac{1}{\sqrt{2}}|z|^{\frac{1}{2}}|x-y|},
    \end{split}
\end{align}
and thus, the kernel of $T_j (z)$ is bounded by a constant times
\begin{align*}
    |z|^{\frac{n-3}{4}} |x-y|^{-\frac{n-1}{2}} e^{-\frac{1}{\sqrt{2}}|z|^{\frac{1}{2}}|x-y|}.
\end{align*}
By \eqref{x-y range}, we know $|z|^{\frac{1}{2}}|x-y|\geq 2^{j-1}\geq 1$ for $j\geq 1$, and so,
\begin{align*}
    e^{-\frac{1}{\sqrt{2}}|z|^{\frac{1}{2}}|x-y|}\leq C_N (|z|^{\frac{1}{2}}|x-y|)^{-N},\quad \text{for any } N\geq 0.
\end{align*}
This implies that the kernel of $T_j (z)$ is majorized by a constant times
\begin{align*}
    |z|^{\frac{n-3}{4}} |x-y|^{-\frac{n-1}{2}} (|z|^{\frac{1}{2}}|x-y|)^{-N}=|z|^{\frac{n-3}{4}-\frac{N}{2}}|x-y|^{-\frac{n-1}{2}-N},\quad \text{for any } N\geq 0.
\end{align*}
By \eqref{x-y range}, we know $|z|\approx 2^{2j} |x-y|^{-2}$, and so,
\begin{align*}
    |z|^{\frac{n-3}{4}-\frac{N}{2}}|x-y|^{-\frac{n-1}{2}-N}&\approx (2^j)^{\frac{n-3}{2}-N} |x-y|^{2-n}.
\end{align*}
Choosing $N$ sufficiently large, this is bounded by a constant times $(2^j)^{-\alpha_n'}|x-y|^{2-n}$ for some $\alpha_n'>0$, and so, by the proof of Lemma \ref{Lemma Sobolev trace for hypsurf}, the bound for $T_j (z)$ with $\arg z\in \left[-\frac{\pi}{2}, \frac{\pi}{2} \right]$ satisfies a desired bound, that is,
\begin{align}\label{Tj est for small arg z}
    \sum_{j=1}^\infty\|T_j (z)\|_{L^p (\mathbb{R}^n)\to L^q (\Sigma)}\lesssim 1,\quad \text{when } \arg z\in \left[-\frac{\pi}{2}, \frac{\pi}{2}\right] \text{ and } \eqref{Exponent condition for hypersurfaces in Rn} \text{ holds.}
\end{align}

We are left to consider the case where $\arg z\not\in \left[-\frac{\pi}{2}, \frac{\pi}{2}\right]$ for $T_j (z)$. In this case, we need an argument different from the above one, since the computation in \eqref{e to the sqrt of z times |x-y|} does not hold for $\arg z\not\in \left[-\frac{\pi}{2}, \frac{\pi}{2}\right]$. Indeed, if we apply the same argument in \eqref{e to the sqrt of z times |x-y|}, then
\begin{align*}
    |e^{-\sqrt{z}|x-y|}|=e^{-|z|^{\frac{1}{2}}\cos \left(\frac{\arg z}{2} \right)|x-y|}\leq e^{-|z|^{\frac{1}{2}}\left(\cos \frac{\pi}{2}\right)|x-y|}= e^0=1,
\end{align*}
and this is not enough for our computation. In short, if $\arg z\not\in \left[-\frac{\pi}{2}, \frac{\pi}{2}\right]$, then we shall consider the oscillation of $e^{-\sqrt{z}|x-y|}$. Recall that, by construction, for $|x-y|=|(x', h(x'))-y|$ as above,
\begin{align*}
    T_j (z) u(x')&=|z|^{\frac{n-1}{4}-\frac{1}{2}}\int_{\mathbb{R}^n} e^{-\sqrt{z}|x-y|} \frac{a(x, y, z)}{|x-y|^{\frac{n-1}{2}}} \beta(2^{-j}|z|^{\frac{1}{2}}|x-y|) u(y)\:dy \\
    &=|z|^{\frac{n-1}{4}-\frac{1}{2}}\int_{\mathbb{R}^n} e^{-i\mathrm{Im}\sqrt{z}|x-y|} e^{-\mathrm{Re}\sqrt{z}|x-y|} \frac{a(x, y, z)}{|x-y|^{\frac{n-1}{2}}} \beta(2^{-j}|z|^{\frac{1}{2}}|x-y|) u(y)\:dy.
\end{align*}
Since the kernel of $T_j (z)$ is supported where $2^{j-1}|z|^{-\frac{1}{2}}\leq |x-y|\leq 2^j |z|^{-\frac{1}{2}}$ by \eqref{x-y range}, we may assume that $u$ is supported in a ball of radius $2^j |z|^{-\frac{1}{2}}$. After a possible translation, without loss of generality, for the purpose of proving \eqref{WTS in S2}, we may assume that $u$ is supported in a ball of radius $2^j |z|^{-\frac{1}{2}}$ centered at $0$, i.e.,
\begin{align}\label{support of u}
    \mathrm{supp}(u)\subset \{y\in \mathbb{R}^n: |y|\leq 2^j |z|^{-\frac{1}{2}}\}.
\end{align}
If we use the change of variables
\begin{align*}
    x=2^j |z|^{-\frac{1}{2}} X,\quad y=2^j |z|^{-\frac{1}{2}} Y,
\end{align*}
then we can write
\begin{align}\label{Tj Tj tilde relation}
    T_j (z) u(x)=|z|^{-1}(2^j)^{\frac{n+1}{2}} \Tilde{T}_j (z) u_j (X),
\end{align}
where
\begin{align}\label{Tj tilde setup}
    \begin{split}
        & u_j (Y)=u(2^j |z|^{-\frac{1}{2}} Y), \\
        & a_j (X, Y, z)=e^{-2^j \cos\left(\frac{\arg z}{2}\right)|X-Y|}\frac{a(2^j|z|^{-\frac{1}{2}}X, 2^j|z|^{-\frac{1}{2}}Y, z)}{|X-Y|^{\frac{n-1}{2}}}\beta(|X-Y|), \\
        & \Tilde{T}_j (z) u_j (X')=\int_{\mathbb{R}^n} e^{-i 2^j \sin\left(\frac{\arg z}{2} \right)|X-Y|} a_j (X, Y, z) u_j (Y)\:dY, \\
        & X=(X', h(X'))\in \mathbb{R}^{n-1}\times \mathbb{R},\quad |X-Y|=|(X', h(X'))-Y|.
    \end{split}
\end{align}
We note that we can focus on $|Y|\lesssim 1$ by \eqref{support of u}, and $|X-Y|\approx 1$ by the support property of the term $\beta(|X-Y|)$. Since $\sin\left(\frac{\arg z}{2}\right)\approx 1$ for $\arg z\not\in\left[-\frac{\pi}{2}, \frac{\pi}{2}\right]$, abusing notations, $\Tilde{T}_j (z) u_j (X')$ can be thought of as
\begin{align}\label{Abusing notation for Tj tilde opr setup}
    \int_{\mathbb{R}^n} e^{-2^j |X-Y|} a_j (X, Y, z) u_j (Y)\:dY,\quad X=(X', h(X'))\in \mathbb{R}^{n-1}\times \mathbb{R}.
\end{align}

Since we are computing estimates for $T_j (z)$ and $\Tilde{T}_j (z)$ locally, abusing notations, we can write
\begin{align}\label{Tj Tj tilde simplify notations}
    \begin{split}
        & \|T_j (z)\|_{L^p (M)}, \|T_j (z)\|_{L^q (\Sigma)}, \|\Tilde{T}_j (z)\|_{L^p (M)}, \text{ and } \|\Tilde{T}_j (z)\|_{L^q (\Sigma)}, \quad\text{as} \\
        & \|T_j (z)\|_{L^p (\mathbb{R}^n)}, \|T_j (z) \|_{L^q (\mathbb{R}^k)}, \|\Tilde{T}_j (z)\|_{L^p (\mathbb{R}^n)}, \text{ and } \|\Tilde{T}_j (z)\|_{L^q (\mathbb{R}^k)}, \text{ respectively,} \\
        & \text{where } n=\dim M \text{ and } k=\dim \Sigma=n-1.
    \end{split}
\end{align}
One reason why we are abusing notations here is that we are in local coordinates as in \eqref{Hypsurf param in Rn}. We are also making use of the change of variables for $T_j (z)$ and $\Tilde{T}_j (z)$. Since $T_j (z)$ is defined for the variable $x'$ and $\Tilde{T}_j (z)$ is defined for the variable $X'$, to distinguish this difference, we write, for $n=\dim M$,
\begin{align}\label{Tj Tj tilde simplified notations}
    \begin{split}
        & \|T_j (z) u\|_{L_y^p (\mathbb{R}^n)}=\left(\int_{\mathbb{R}^n} |T_j (z) u(y)|^p \:dy \right)^{\frac{1}{p}},\quad \|\Tilde{T}_j (z) u_j\|_{L_Y^p (\mathbb{R}^n)}=\left(\int_{\mathbb{R}^n} |\Tilde{T}_j (z) u_j(Y)|^p\:dY \right)^{\frac{1}{p}}, \\
        & \|T_j (z) u\|_{L_{x'}^p (\mathbb{R}^{n-1})}=\left(\int_{\mathbb{R}^{n-1}} |T_j (z) u(x')|^p \:dx' \right)^{\frac{1}{p}},\quad \|\Tilde{T}_j (z) u_j\|_{L_{X'}^p (\mathbb{R}^{n-1})}=\left(\int_{\mathbb{R}^{n-1}} |\Tilde{T}_j (z) u_j(X')|^p\:dX' \right)^{\frac{1}{p}}.
    \end{split}
\end{align}
Here, $\|\Tilde{T}_j (z) u_j\|_{L_{X'}^p (\mathbb{R}^{n-1})}$ (similarly $\|u_j\|_{L_X^p (\mathbb{R}^{n-1})}$ and $\|T_j (z) u\|_{L_x^p (\mathbb{R}^{n-1})}$ as well) can be interpreted as
\begin{align*}
    \left(\int_{\mathbb{R}^{n-1}} |\Tilde{T}_j (z) u_j(X', h(X'))|^p \:dX' \right)^{\frac{1}{p}},\quad \text{for } X=(X', h(X'))\in \mathbb{R}^{n-1}\times \mathbb{R}.
\end{align*}

With this in mind, by a straightforward computation, one can see that
\begin{align}\label{Operator norm comparisons for Euclidean hypersurfaces}
    \begin{split}
        & \|T_j (z)\|_{L_y^p (\mathbb{R}^n)\to L_{x'}^q (\mathbb{R}^{n-1})}=|z|^{\frac{1}{2}\left(\frac{n}{p}-\frac{n-1}{q}-2 \right)} (2^j)^{\frac{n+1}{2}-\left(\frac{n}{p}-\frac{n-1}{q} \right)} \|\Tilde{T}_j (z)\|_{L_Y^p (\mathbb{R}^n)\to L_{X'}^q (\mathbb{R}^{n-1})},\quad \text{and} \\
        & \|T_j (z)\|_{L_y^p (\mathbb{R}^n)\to L_{x'}^q (\mathbb{R}^{n-1})}=(2^j)^{\frac{n-3}{2}} \|\Tilde{T}_j (z)\|_{L_Y^p (\mathbb{R}^n)\to L_{X'}^q (\mathbb{R}^{n-1})}\quad \text{if} \quad \frac{n}{p}-\frac{n-1}{q}=2.
    \end{split}
\end{align}
We now need some estimates for $T_j (z)$, but we note that this is already done in \cite[\S 5]{BlairPark2024RestrictionOfSchrodingerEigenfunctions}. Indeed, the operators $T_j (z)$ and $\tilde{T}_j (z)$ here for $\arg z\not\in\left[-\frac{\pi}{2}, \frac{\pi}{2}\right]$ are the analogues of $S_j$ and $\tilde{S}_j$ in \cite[\S 5]{BlairPark2024RestrictionOfSchrodingerEigenfunctions}. That is, if we take the polar coordinates at $Y$, then the integral expression of $\tilde{T}_j(z)u_j(X')$ in \eqref{Tj tilde setup} has the phase function
\begin{align*}
    (X, Y)=(X, (r, \omega))\mapsto -|X-Y|=-|X-r\omega|,\quad Y=r\omega,
\end{align*}
is the Euclidean distance function for the polar coordinates. Since \cite[Theorem 1.3]{Hu2009lp} showed that this distance function satisfies the conditions in Greenleaf and Seeger \cite[Theorem 2.1-2.2]{GreenleafSeeger1994fourier}, by \cite[Theorem 2.1-2.2]{GreenleafSeeger1994fourier}, we can find $L^q$ estimates of $\tilde{T}_j(z) u_j$. In fact, \cite{Hu2009lp} showed that Riemannian distance functions satisfy the conditions in \cite{GreenleafSeeger1994fourier}, and Euclidean distance functions are the special cases of the Riemannian ones. Thus, by the proofs of \cite[Lemma 5.3-5.5]{BlairPark2024RestrictionOfSchrodingerEigenfunctions}, we have the following lemmas.

\begin{lemma}\label{Lemma Lp0 to Lq0 estimate for hypersurfaces in Rn}
    If $q_0=\frac{2n^2-2n+1}{(n-1)(n-2)}$ and $\frac{n}{p_0}-\frac{n-1}{q_0}=2$, then for $j\geq 1$
    \begin{align}\label{Lp0 to Lq0 estimate for hypersurfaces in Rn}
        \|T_j (z) u\|_{L_{x'}^{q_0}(\mathbb{R}^{n-1})}\lesssim (2^j)^{-\frac{3n-1}{2(2n^2-2n+1)}} \|u\|_{L_y^{p_0} (\mathbb{R}^n)}.
    \end{align}
    We also have that
    \begin{align}\label{Estimate for critical exp}
        \|T_j (z)u\|_{L_{x'}^{\frac{2n}{n-1}}(\mathbb{R}^{n-1})}\lesssim |z|^{-\frac{2n+1}{4n}}(2^j)^{\frac{1}{2}} \|u\|_{L_y^2 (\mathbb{R}^n)}.
    \end{align}
\end{lemma}

\begin{lemma}\label{Lemma upper q uniform estimate for hypersurfaces in Rn}
    For $j\geq 1$, we have
    \begin{align}\label{Upper q uniform estimate for hypersurfaces in Rn}
        \|T_j (z)u\|_{L_{x'}^{\frac{2(n-1)}{n-3}} (\mathbb{R}^{n-1})}\lesssim \|u\|_{L_y^{\frac{2n}{n+1}} (\mathbb{R}^n)}.
    \end{align}
\end{lemma}

\begin{lemma}\label{Lemma lower q uniform estimate for hypersurfaces in Rn}
    Suppose $\frac{n}{p}-\frac{n-1}{q}=2$. For $j\geq 1$, we have that
    \begin{align}\label{Lower q uniform estimate for hypersurfaces in Rn for n=3, 4, 5}
        \|T_j (z)u\|_{L_{x'}^{\frac{2(n-1)^2}{n^2-3n+4}} (\mathbb{R}^{n-1})}\lesssim \|u\|_{L_y^{\frac{2(n-1)}{n+1}} (\mathbb{R}^n)},\text{ when } n=3, 4, 5,
    \end{align}
    and
    \begin{align}\label{Lower q uniform estimate for hypersurfaces in Rn for n geq 6}
        \|T_j (z)u\|_{L_{x'}^{\frac{2n^2-5n+4}{n^2-4n+8}} (\mathbb{R}^{n-1})}\lesssim \|u\|_{L_y^{\frac{2n^2-5n+4}{n^2-n+2}} (\mathbb{R}^n)},\text{ when } n\geq 6.
    \end{align}
\end{lemma}

By Lemma \ref{Lemma Lp0 to Lq0 estimate for hypersurfaces in Rn}-\ref{Lemma lower q uniform estimate for hypersurfaces in Rn}, if $\frac{n}{p}-\frac{n-1}{q}=2$ and $n\in \{3, 4, 5\}$, then
\begin{align*}
    & \|T_j (z)u\|_{L^q (\mathbb{R}^{n-1})}\lesssim (2^j)^{-\frac{3n-1}{2(2n^2-2n+1)}}\|u\|_{L^p (\mathbb{R}^n)},\quad \text{if } q=\frac{2n^2-2n+1}{(n-1)(n-2)}, \\
    & \|T_j (z)u\|_{L^q (\mathbb{R}^{n-1})}\lesssim \|u\|_{L^p (\mathbb{R}^n)},\quad \text{if } q=\frac{2(n-1)}{n-3} \text{ or } q=\frac{2(n-1)^2}{n^2-3n+4}. \\
\end{align*}
By interpolation, for some $\alpha_n (p, q)>0$, if $\frac{n}{p}-\frac{n-1}{q}=2$ and $n\in \{3, 4, 5\}$, then
\begin{align}\label{Interpolation n=3, 4 for hypsurf in Euclidean}
    \begin{split}
        \|T_j (z)u\|_{L^q (\mathbb{R}^{n-1})}\lesssim (2^j)^{-\alpha_n (p, q)} \|u\|_{L^p (\mathbb{R}^n)},\quad \text{ when } j\geq 1 \text{ and } \frac{2(n-1)^2}{n^2-3n+4}<q<\frac{2(n-1)}{n-3}.
    \end{split}
\end{align}
If $\frac{n}{p}-\frac{n-1}{q}=2$ and $n\geq 6$, then by Lemma \ref{Lemma Lp0 to Lq0 estimate for hypersurfaces in Rn}-\ref{Lemma lower q uniform estimate for hypersurfaces in Rn}
\begin{align*}
    & \|T_j (z)u\|_{L^q (\mathbb{R}^{n-1})}\lesssim (2^j)^{-\frac{3n-1}{2(2n^2-2n+1)}}\|u\|_{L^p (\mathbb{R}^n)},\quad \text{if } q=\frac{2n^2-2n+1}{(n-1)(n-2)}, \\
    & \|T_j (z)u\|_{L^q (\mathbb{R}^{n-1})}\lesssim \|u\|_{L^p (\mathbb{R}^n)},\quad \text{if } q=\frac{2(n-1)}{n-3} \text{ or } q=\frac{2n^2-5n+4}{n^2-4n+8}. \\
\end{align*}
By interpolation, for some $\alpha_n (p, q)>0$, if $\frac{n}{p}-\frac{n-1}{q}=2$ and $n\geq 6$, then
\begin{align}\label{Interpolation n geq 5 for hypsurf in Euclidean}
    \|T_j (z)u\|_{L^q (\mathbb{R}^{n-1})}\lesssim (2^j)^{-\alpha_n (p, q)} \|u\|_{L^p (\mathbb{R}^n)},\quad \text{when } j\geq 1 \text{ and } \frac{2n^2-5n+4}{n^2-4n+8}\leq q<\frac{2(n-1)}{n-3}.
\end{align}
By \eqref{Interpolation n=3, 4 for hypsurf in Euclidean} and \eqref{Interpolation n geq 5 for hypsurf in Euclidean}, we have \eqref{WTS in S2} for some $\alpha_n (p, q)>0$, and this completes the proof of \eqref{Unif Sobolev est for hypsurf}.

\section{Proof of Theorem \ref{Thm Manifold resolvent restriction estimates for hypersurface}}\label{S: Proof of Thm for manifolds}
In this section, we prove \eqref{Manifold resolvent est for hypsurf}. After scaling metric if needed, we may assume $\delta=1$ where $\delta$ is as in \eqref{z lambda mu setup}. Without loss of generality, we may assume $\mu\geq 1$. Similar arguments work for $\mu\leq -1$. We follow the arguments of \cite[\S 2]{BourgainShaoSoggeYao2015Resolvent} in this section. Recall that we write $P=\sqrt{-\Delta_g}$. By \cite[(2.3)]{BourgainShaoSoggeYao2015Resolvent} and \cite[(3.16)]{BlairHuangSireSogge2022UniformSobolev}, we can write
\begin{align*}
    (-\Delta_g-(\lambda+i\mu)^2)^{-1}=\frac{i}{\lambda+i\mu}\int_0^\infty e^{i\lambda t} e^{-\mu t} (\cos tP)\:dt,\quad \lambda, \mu\geq 1.
\end{align*}
We fix a function $\rho\in C^\infty (\mathbb{R})$ satisfying
\begin{align*}
    \rho(t)=\begin{cases}
    1, & \text{if } t\leq \epsilon_0/2, \\
    0, & \text{if } t\geq \epsilon_0,
    \end{cases}
\end{align*}
where $0<\epsilon_0\ll 1$. We write
\begin{align*}
    (-\Delta_g-(\lambda+i\mu)^2)^{-1}=S_\lambda+W_\lambda,
\end{align*}
where
\begin{align}\label{S lambda W lambda setup}
    \begin{split}
        & S_\lambda=\frac{i}{\lambda+i\mu}\int_0^\infty \rho(t)e^{i\lambda t} e^{-\mu t} (\cos tP)\:dt, \\
        & W_\lambda=\frac{i}{\lambda+i\mu}\int_0^\infty (1-\rho(t))e^{i\lambda t} e^{-\mu t} (\cos tP)\:dt.
    \end{split}
\end{align}

Let $S_\lambda$ be as in \eqref{S lambda W lambda setup}. We first recall that \cite[Proposition 3.3]{BlairPark2024RestrictionOfSchrodingerEigenfunctions} gives the following $L^p (M) \to L^q (\Sigma)$ estimates.

\begin{proposition}[\cite{BlairPark2024RestrictionOfSchrodingerEigenfunctions}]\label{Prop Mfld local resolvent est for hypsurf}
    Let $(M, g)$ be a compact Riemanian manifold without boundary of dimension $n\geq 3$, and $\Sigma$ be a hypersurface of $M$. Then we have
    \begin{align*}
        \|S_\lambda f\|_{L^q (\Sigma)}\leq C_\delta \|f\|_{L^p (M)},
    \end{align*}
    where $\lambda, \mu\in \mathbb{R}, \lambda, \mu\geq 1$, and
    \begin{align}\label{p, q conditions of resolvent est for hypsurf in mfld}
        \begin{split}
            \frac{n}{p}-\frac{n-1}{q}=2, \quad \text{and} \quad \begin{cases}
            \frac{2(n-1)^2}{n^2-3n+4}<q<\frac{2(n-1)}{n-3} & \text{when } n=3, 4, 5, \\
            \frac{2n^2-5n+4}{n^2-4n+8}<q<\frac{2(n-1)}{n-3} & \text{when } n\geq 6.
            \end{cases}
        \end{split}
    \end{align}
\end{proposition}

For the estimate of $W_\lambda$ defined in \eqref{S lambda W lambda setup}, we follow the argument in \cite[Lemma 2.3]{BourgainShaoSoggeYao2015Resolvent} to show its analogue.

\begin{lemma}\label{Lemma Truncated resolvent est for hypsurf}
    Let $(M, g)$ be a compact Riemannian manifold of dimension $3\leq n\leq 12$, and $\Sigma$ be a hypersurface of $M$. For $\eta\in C(\overline{\mathbb{R}_+})$, we define the operator $\eta_k (P)$ by
    \begin{align*}
        \eta_k (P)f=\sum_{\lambda_j\in [k-1, k)} \eta(\lambda_j) E_j f,\quad f\in C^\infty (M),\quad k=1, 2, 3, \cdots,
    \end{align*}
    where the $E_j$ are the projections onto the $j$th eigenspace, i.e.,
    \begin{align*}
        E_j f(x)=\left(\int_M f(y)e_j (y)\:dV_g \right)e_j (x),
    \end{align*}
    where the $e_\lambda$ are as in \eqref{e lambda eigfcn}. Then we have
    \begin{align}\label{eta k (P) resulting estimate}
        \|\eta_k (P) f\|_{L^q (\Sigma)}\lesssim k \Big(\sup_{\tau\in [k-1, k]} |\eta(\tau)| \Big) \|f\|_{L^p (M)},\quad \text{if \eqref{(p, q) condition for hypsurf in mfld} holds.}
    \end{align}
\end{lemma}

This lemma is the reason why we focus on $3\leq n\leq 12$ in \eqref{Manifold resolvent est for hypsurf}.

\begin{proof}[Proof of Lemma \ref{Lemma Truncated resolvent est for hypsurf}]
    Let $\chi_\lambda$ be the spectral projection operator, i.e.,
    \begin{align*}
        \chi_\lambda f=\sum_{\lambda_j\in [k-1, k)} E_j f.
    \end{align*}
    By \cite{Sogge1988concerning} (see also \cite[Theorem 5.1.1]{Sogge1993fourier}) and \cite{BurqGerardTzvetkov2007restrictions} (see also \cite{Hu2009lp}), we know that, for $\lambda\geq 1$ and $q\geq 2$,
    \begin{align}\label{Spec bound for chi lambda}
        \|\chi_\lambda \|_{L^p (M)\to L^2 (M)}\lesssim \lambda^{\sigma(p')},\quad \|\chi_\lambda\|_{L^2 (M)\to L^q (\Sigma)}\lesssim \lambda^{\mu(q, n-1)},
    \end{align}
    where
    \begin{align}\label{mu sigma setups}
        \begin{split}
            & \mu (q, n-1)=\begin{cases}
                \frac{n-1}{4}-\frac{n-2}{2q}, & \text{if } 2\leq q\leq \frac{2n}{n-1}, \\
                \frac{n-1}{2}-\frac{n-1}{q}, & \text{if } \frac{2n}{n-1}\leq q\leq +\infty,
            \end{cases} \\
            & \sigma(q)=\begin{cases}
                n\left(\frac{1}{2}-\frac{1}{q} \right)-\frac{1}{2}, & \text{if } q\geq \frac{2(n+1)}{n-1}, \\
                \frac{n-1}{2}\left(\frac{1}{2}-\frac{1}{q} \right), & \text{if } 2\leq q\leq \frac{2(n+1)}{n-1}.
            \end{cases}
        \end{split}
    \end{align}
    We note that
    \begin{align*}
        \sigma(p')=\begin{cases}
            n\left(\frac{1}{p}-\frac{1}{2} \right)-\frac{1}{2}, & \text{if } 1\leq p\leq \frac{2(n+1)}{n+3}, \\
            \frac{n-1}{2}\left(\frac{1}{p}-\frac{1}{2} \right), & \text{if } \frac{2(n+1)}{n+3}\leq p\leq 2.
        \end{cases}
    \end{align*}
    Since $\eta_k(P)=\chi_k\circ \eta_k (P)\circ \chi_k$ for $k=1, 2, 3, \cdots$, by \eqref{Spec bound for chi lambda}
    \begin{align}\label{eta k (P) est computation}
        \begin{split}
            \|\eta_k (P)f\|_{L^q (\Sigma)}&\lesssim k^{\mu (q, n-1)} \|\eta_k (P) \chi_k f\|_{L^2 (M)} \\
            &\lesssim k^{\mu (q, n-1)} \Big(\sup_{\tau\in [k-1, k]} |\eta(\tau)|\Big) \|\chi_k f\|_{L^2 (M)} \\
            &\lesssim k^{\mu (q, n-1)+\sigma(p')} \Big(\sup_{\tau\in [k-1, k]} |\eta(\tau)| \Big) \|f\|_{L^p (M)},\quad 1\leq p\leq 2\leq q.
        \end{split}
    \end{align}
    We now want to find $(p, q)$ satisfying \eqref{p, q conditions of resolvent est for hypsurf in mfld} and
    \begin{align*}
        \mu(q, n-1)+\sigma(p')\leq 1,\quad \text{for } 1\leq p\leq 2\leq q.
    \end{align*}
    We note that when $\frac{n}{p}-\frac{n-1}{q}=2$, we have $p=\frac{nq}{2q+n-1}$, and so, $p=\frac{nq}{2q+n-1}\leq \frac{2(n+1)}{n+3}$ if and only if $q\leq \frac{2(n-1)(n+1)}{n^2-n-4}$. This gives us that if $p=\frac{nq}{2q+n-1}$, then
    \begin{align*}
        \sigma (p')=\begin{cases}
            \frac{n-1}{q}-\frac{n-3}{2}, & \text{for } q\leq \frac{2(n-1)(n+1)}{n^2-n-4}, \\
            \frac{n-1}{2}\left(\frac{2q+n-1}{nq}-\frac{1}{2} \right), & \text{for } q\geq \frac{2(n-1)(n+1)}{n^2-n-4}.
        \end{cases}
    \end{align*}
    With this in mind, by straightforward computation, we have
    \begin{align}\label{mu plus sigma leq 1 computation}
        \begin{split}
            \mu(q, n-1)+\sigma(p')\leq 1,\quad \text{if} \quad \frac{n}{p}-\frac{n-1}{q}=2,\; 3\leq n\leq 12, \text{ and } \frac{2n}{n-1}\leq q\leq \frac{2(n-1)(n+1)}{n^2-n-4}.
        \end{split} 
    \end{align}
    We leave this computation to the readers. Combining \eqref{eta k (P) est computation} and \eqref{mu plus sigma leq 1 computation}, we have \eqref{eta k (P) resulting estimate}, completing the proof of this lemma.
\end{proof}

We are now ready to prove Theorem \ref{Thm Manifold resolvent restriction estimates for hypersurface}. Recall that we already set
\begin{align*}
    W_\lambda=(-\Delta_g-z^2)^{-1}-S_\lambda.
\end{align*}
By Proposition \ref{Prop Mfld local resolvent est for hypsurf}, we would have the desired inequality \eqref{Manifold resolvent est for hypsurf} if we could show that
\begin{align}\label{WTS Mfld W lambda est}
    \begin{split}
        \|W_\lambda f\|_{L^q (\Sigma)}\lesssim \|f\|_{L^p (M)},\quad \lambda, \mu\in \mathbb{R},\; \lambda, |\mu|\geq 1, \quad \text{where \eqref{(p, q) condition for hypsurf in mfld} holds.}
    \end{split}
\end{align}
By construction, the multiplier for $W_\lambda$ is the function
\begin{align*}
    w_{\lambda, \mu}(\tau)=\frac{\mathrm{sgn}\mu}{i(\lambda+i\mu)} \int_0^\infty (1-\rho(t)) e^{i(\mathrm{sgn}\mu)\lambda t} e^{-|\mu|t} \cos t\tau\:dt.
\end{align*}
Since $1-\rho(t)\equiv 0$ near the origin, if we use the formula $\cos t\tau=\frac{1}{2}(e^{it\tau}+e^{-it\tau})$, then integration by parts gives us that
\begin{align*}
    |w_{\lambda, \mu}(\tau)|\leq C_N \lambda^{-1} ((1+|\lambda-\tau|)^{-N}+(1+|\lambda+\tau|)^{-N}),\quad \tau, \lambda, \mu\in \mathbb{R},\; \lambda\geq 1,\; |\mu|\geq 1.
\end{align*}
By this and Lemma \ref{Lemma Truncated resolvent est for hypsurf}, if \eqref{(p, q) condition for hypsurf in mfld} holds, then we have
\begin{align}\label{W lambda est for hypsurf in mfld}
    \|W_\lambda \|_{L^p (M)\to L^q (\Sigma)}\lesssim \sum_{k=1}^\infty k \lambda^{-1} (1+|\lambda-k|)^{-3}\lesssim 1,
\end{align}
which proves \eqref{WTS Mfld W lambda est}. This proves \eqref{Manifold resolvent est for hypsurf}.

\bibliographystyle{amsalpha}
\bibliography{references}

\end{document}